\documentclass[12pt]{article}
\usepackage[T1]{fontenc}
\usepackage{pgf}
\usepackage{authblk}
\usepackage[latin1]{inputenc}
\usepackage[english]{babel}
\usepackage{amsmath,enumerate}
\usepackage{amssymb}
\usepackage{amscd}
\usepackage{tabularx}
\usepackage{vmargin}
 \usepackage{float}
\usepackage{amsthm}
\usepackage{url}
\usepackage[usenames,dvipsnames]{pstricks}
\usepackage{epsfig}
\usepackage{pst-grad} 
\usepackage{pst-plot} 
\usepackage{epsfig}
\usepackage{subfigure}
\usepackage{graphicx}

\newtheorem{theorem}{Theorem}[section]
\newtheorem{lemma}[theorem]{Lemma}
\newtheorem{conjecture}[theorem]{Conjecture}

\newtheorem{observation}{Observation}
\newtheorem{remark}{Remark}
\newtheorem{claim}[theorem]{Claim}

\newtheorem{corollary}[theorem]{Corollary}

\newtheorem{question}{Question}
\newtheorem{proposition}[theorem]{Proposition}
\newtheorem*{TheoNonNum}{Theorem}
\newtheorem*{propoNoNum}{Proposition}
\begin{document}
\title{Injective edge-coloring of sparse graphs}
\author[1,4]{Baya Ferdjallah}
\author[1,5]{Samia Kerdjoudj }

\author[2]{Andr\'e Raspaud}
\affil[1]{LIFORCE, Faculty of Mathematics, USTHB, BP 32 El-Alia, Bab-Ezzouar 16111, Algiers, Algeria}
\affil[2]{LaBRI, Universit\'e de Bordeaux, 351 cours de la Lib\'eration, 33405 Talence Cedex, France}
\affil[4]{Universit\'e de Boumerd\`es, Avenue de l'ind\'ependance, 35000, Boumerd\`es, Algeria}
\affil[5]{Universit\'e de Blida 1, Route de Soum\^aa BP 270, Blida, Algeria}

\maketitle

\begin{abstract}

An injective edge-coloring $c$ of a graph $G$ is an edge-coloring such that if $e_1$, $e_2$, and $e_3$ are three consecutive edges in $G$ (they are consecutive if they form a path  or a cycle of length three), then $e_1$ and $e_3$ receive different colors. The minimum  integer $k$ such that, $G$ has an injective edge-coloring with $k$ colors,  is called the injective chromatic index of $G$ ($\chi'_{\textrm{inj}}(G)$). This parameter was introduced by Cardoso et \textit{al.}  \cite{CCCD} motivated by the Packet Radio Network problem. They proved that computing $\chi'_{\textrm{inj}}(G)$ of a graph $G$ is NP-hard. 

We give new upper bounds for this parameter and we present the relationships of the injective edge-coloring with other colorings of graphs.
The obtained general bound gives 8 for the injective chromatic index of a  subcubic graph. If the graph is subcubic bipartite we improve this last bound. We prove that a subcubic bipartite graph has an injective chromatic index bounded by $6$. We also prove that if $G$ is a subcubic graph with maximum average degree less than $\frac{7}{3}
$ (resp. $\frac{8}{3}
$, $3$), then $G$ admits an injective edge-coloring with at most 4 (resp. $6$, $7$) colors. Moreover, we establish a tight upper bound for subcubic outerplanar graphs.

\end{abstract}

%

\section{Introduction}

All the graphs we consider are finite and simple. For a graph $G$, we  denote by $V(G)$, $E(G)$, $\delta(G)$, and $\Delta(G)$ its vertex set, edge set, minimum degree, and maximum degree, respectively. We denote by $d_G(v)$ the degree of a vertex $v$ in $G$. If $G$ is clear from the context, then we may omit the subscript. The girth of a graph $G$ is the length of a shortest cycle contained in the graph. If two edges $e$ and $f$  of a graph $G$ are the two extremities of a path of length $3$, we will say that $e$ and $f$ are at distance $2$.\\

A proper vertex (respectively, edge) coloring of $G$ is an assignment of colors to the vertices (respectively, edges) of $G$ such that no two adjacent vertices (respectively, edges) receive the same color. 

Let $G= (V, E)$ be a graph and  $c$ be a vertex
$k$-coloring of $G$. $c$ is  a function $c : V \longrightarrow  \{1, \ldots, k\}$.  We say that the vertex coloring of $G$ is {
injective} if its restriction to the neighborhood of any vertex
is injective. The { injective chromatic number} $\chi_{\rm {inj}}(G)$ of
a graph $G$ is the least $k$ such that there is an injective $k$-coloring. Injective coloring of graphs was introduced by Hahn
et. al in \cite{hss}. Note that an injective coloring is
not necessarily a proper coloring.\\
 Three edges $e_1$, $e_2$, and $e _3$ in a graph $G$ are consecutive if they form a path (in this order) or a cycle of length three. An injective edge-coloring of a graph $G = (V, E)$ is a coloring $c$ of the edges of $G$ such that if $e_1$, $e_2$, and $e_3$ are consecutive edges in $G$, then $c(e_1 ) \neq c(e_3 )$.
In other words, every two
edges at distance exactly $2$ or belonging to a triangle do not use the same color.
The injective chromatic index of $G$, denoted by $\chi'_{\rm {inj}}(G)$, is the minimum number of colors needed for an injective edge-coloring of $G$. Note that an injective edge-coloring is not necessarily a proper edge-coloring. This notion was introduced in 2015 
by Cardoso et \textit{al.} \cite{CCCD} motivated by a Packet Radio Network problem. Two adjacent entities in a network must use different frequencies to send  messages to their other neighbors. What is the minimum number of frequencies needed?\\
When we consider proper colorings, it is well known that the edge chromatic number of a graph $G$ 
 is equal to the vertex chromatic number of its line graph $L(G)$, $\chi'(G) = \chi(L(G))$. 
For the injective coloring it is not always true, that  $\chi'_{\rm {inj}}(G) = \chi_{\rm {inj}}(L(G))$.  For the star $K_{1,n}$ we have $\chi'_{\rm {inj}}(K_{1,n})=1$ and $\chi_{\rm {inj}}(L(K_{1,n})=n$.

In \cite{CCCD}, it is proved that computing $\chi'_{\rm {inj}}(G)$ of a graph $G$ is NP-hard and the authors  gave the injective chromatic index of some classes of graphs (paths, cycles, wheels,  Petersen graph and complete bipartite graphs). 
\begin{proposition}\cite{CCCD}\label{cyclechaine}
Let $P_n$ (resp. $C_n$ ) be a path (resp. a cycle) with $n$ vertices. Then 
\begin{enumerate}
\item $\chi'_{\rm {inj}}(P_n) = 2$, for $n\geq 4$.
\item $\chi'_{{\rm {inj}}}(C_n)=\begin{cases} 2& {\rm if }~ n \equiv 0 \pmod 4;\\ 3& {\rm otherwise}.\end{cases}$

\end{enumerate} 
\end{proposition}

For the case of trees, they gave the following bound:

\begin{proposition}\label{CCCD}\cite{CCCD}
 For any tree of order $n\geq 2$, $1\leq \chi_{\rm{inj}}'(G)\leq 3$.
\end{proposition}
 The injective chromatic index is a parameter  interesting by it-self and determine bounds for this parameter and the injective chromatic index of graphs with maximum degree 3 graphs is already a good challenge. It is our motivation. 
 
In 2019, Axenovich et  \textit{al.} \cite{axia2019} introduced the notion of induced arboricity of a graph $G$, denoted by $\text{ia}(G)$, it is the smallest number $k$ such that the edges of $G$ can be covered with $k$ induced forests in $G$. 
\\
For a graph parameter $\rm p$ and a class of graphs $F$, let  $\text{p}( F ) = \sup \{\text{p}( G ) : G \in 
F\}$.\\
In \cite{axia2019}, the authors provide bounds on $\text{ia}(F)$ when $F$ is the class of planar graphs, the class of $d$-degenerate graphs, or the class of graphs having tree-width at most $d$. They prove that if $F$ is the class of planar graphs, then $8\leq \text{ia}( F ) \leq 10$. 
They  establish similar results for the induced star
arboricity of classes of graphs.
The star arboricity of a graph $G$,  denoted by $\text{isa}(G)$, is defined as the smallest number of induced star-forests covering the edges of $G$. 
\\

They prove that if $F$ is the class of planar graphs, then $18\leq \text{isa}(F)\leq 30$. \\

 By applying the Proposition 2.2 of \cite{CCCD} we have the following easy observation:

\begin{observation}
For any graph $G$, we have $\chi'_{\rm {inj}}(G)={\rm isa}(G)$.
\end{observation}

 In what follows we will use the terminology of injective edge-coloring, as before, we find that it is easier to handle colors. \\
 
We now present the relationships of the injective edge-coloring with other colorings of graphs and the deduced bounds for  the injective chromatic index. \\

 \subsection*{Strong edge-coloring}
 A strong $k$-edge-coloring of a graph $G$ is an assignment of $k$ colors to the edges of $G$ in such a way that any two edges at distance at most two are assigned distinct colors. The minimum number of colors for which a strong edge-coloring of $G$ exists is the strong chromatic index of $G$, denoted $\chi'_{s}(G)$. This notion was introduced
by Fouquet and Jolivet \cite{FJ1,FJ2}. In 1985, Erd\H os and Ne\v set\v ril  conjectured that the strong chromatic index of a graph is at most $\frac{5}{4}\Delta^2(G)$ if $\Delta(G)$ is even and $\frac{1}{4}(5\Delta^2(G)-2\Delta(G) + 1)$ if $\Delta(G)$ is odd. Andersen \cite{A} proved the conjecture for $\Delta(G)= 3$. When $\Delta$ is large enough, Molloy and Reed \cite{MR} showed that $\chi'_s(G)\leq  1.998\Delta^2(G)$. This result was improved by Bruhn and Joos \cite{BJ18} who proved  $\chi'_s(G)\leq  1.93\Delta^2(G)$ for $\Delta$ large enough. 
 Let us give an immediate remark. For any graph $G$, we have:
$$\chi_{\rm inj}'(G)\leq \chi_{s}'(G)$$

 \subsection*{Acyclic coloring}
A proper vertex coloring of a graph $G$ is acyclic if there is no bicolored cycle in $G$. In other words, the union of any two color classes induces a forest. The acyclic chromatic
number, denoted by $\chi_a(G)$, of a graph $G$ is the smallest integer $k$ such that $G$ has an acyclic $k$-coloring. This notion was introduced by Gr\"unbaum \cite{Gr} in 1973 and studied by Mitchem \cite{Mi}, Albertson and Berman \cite{AB}, and Kostochka \cite{Ko}.
In 1979, Borodin \cite{Bo} confirmed the conjecture of Gr\"unbaum by proving :
 \begin{theorem}\label{thmBo}\cite{Bo}
Every planar graph is acyclically 5-colorable.
 \end{theorem}
Since  $\chi'_{\rm inj}(G)={\rm isa}(G)$, we can say that the authors of \cite{axia2019} (Theorem 8) proved: 
\begin{theorem}\label{AS1}
Let $G$ be a graph with $\chi_a(G) = k$. Then $\chi_{\rm inj}'(G)\leq \frac{3k(k-1)}{2}$ .
\end{theorem}
By using Theorem \ref{thmBo}, they deduce that, for any planar graph:  $\chi_{\rm inj}'(G)\leq 30$. Moreover they construct an example of planar graph with injective chromatic index equal to 18.\\

Borodin, Kostochka, and  Woodall \cite{BKW} proved that, if $G$ is a planar graph with girth $g\geq 5$ (respectively, $g\geq 7$), then $\chi_a(G)\leq 4$ (respectively, $\chi_a(G)\leq 3$).
\begin{corollary}\label{1.5}
Let $G$ be a planar graph with girth $g$,

\begin{enumerate}
\item If $g\geq5$, then $\chi_{\rm inj}'(G)\leq 18$.
\item If $g\geq7$, then $\chi_{\rm inj}'(G)\leq 9$.
\end{enumerate} 
\end{corollary}
\subsection*{Star coloring}
A star coloring of $G$ is a proper vertex coloring of $G$
such that the union of any two color classes induces a star forest in $G$, i.e. every component of this union is a star. The star chromatic number of $G$, denoted by $\chi_{st}(G)$, is the smallest integer $k$ for which $G$ admits a star coloring with $k$ colors. This notion was first mentioned by Gr\"unbaum \cite{Gr} in 1973 (see also \cite{FRR}). Albertson et al. \cite{ACKKR} proved that for any planar graph $G$, $\chi_{st}(G)\leq 20$. \\
By using the star chromatic number we obtain a bound for the injective chromatic index. 

\begin{theorem}\label{AS2}
Let G be a graph with $\chi_{st}(G) = k$. Then $\chi_{\rm inj}'(G)\leq \frac{k(k-1)}{2}$.
\end{theorem} 

The bound obtained for planar graphs by using Theoerm \ref{AS2} with the result of Albertson et al. \cite{ACKKR}
is not interesting because of the bound obtained by using the acyclic coloring. But in some cases it gives an interesting bound.\\
In \cite{BCMRW} it is proved that, if $G$ is planar and has girth $g\geq 13 $, then $\chi_{st}(G)\leq 4$.

 Theorem \ref{AS2} yields the following corollary which completes Corollary \ref{1.5}:

\begin{corollary}
 If $G$ is a planar graph with girth $g\geq 13$, then $\chi_{\rm inj}'(G)\leq 6$.

\end{corollary}
In the reverse direction we have:

\begin{theorem}\label{Reverse}
Let $G$ be a graph with $\chi_{\rm inj}'(G)=k$. Then $\chi_{st}(G)\leq 2^k$.
\end{theorem}
We  can improve a little bit Theorem \ref{Reverse} when the  graph $G$ has a  minimum degree  at least 2. In this case, if $\chi_{\rm inj}'(G)=k$, then $\chi_{st}(G)\leq 2^k-1$ (see Section $3$).\\

\subsection*{Upper bound}
We give an upper bound for the injective chromatic index of a graph in terms of its maximum degree: 

\begin{theorem}\label{Delta}
Let $G$ be a graph with any maximum degree $\Delta(G)\geq 3$. Then $\chi_{\rm inj}'(G)\leq 2(\Delta(G)- 1)^2.$
\end{theorem}

We suspect that the upper bound of Theorem 1.9 is not  tight. For a lower bound we have only the bound given by the complete graph $K_{\Delta +1}$ of degree $\Delta$, its injective chromatic index is $\frac{(\Delta+1)(\Delta +2)}{2}$.\\
The previous result implies that, for every  subcubic graph $G$, $\chi_{\rm inj}'(G)\leq 8$. 
We propose the following conjecture (see Section 8).
\begin{conjecture}
For every subcubic graph, $\chi_{\rm inj}'(G)\leq 6$. 
\end{conjecture}
\subsection*{Upper Bound for bipartite graphs}
Cardoso et \textit{al.} \cite{CCCD} considered the injective edge-coloring of bipartite graphs. They showed that:
\begin{proposition}\label{Propbiparti}
If $G$ is a bipartite graph with bipartiton $V(G)=V_1\cup V_2$ and $G$ has no isolated vertices, then $\chi_{\rm inj}'(G)\leq\min\{|V_1|,|V_2|\}$.
\end{proposition} 
The above bound is attained for every complete bipartite graph $K_{p,q}$.\\

We give  an other useful bound in terms of maximum degree.
 \begin{proposition}\label{DeltaBiP}
 If $G=(V_1\cup V_2,E)$ is a bipartite graph with maximum degree $\Delta(G)$ and $G$ has no isolated vertices, then $$\chi_{\rm inj}'(G)\leq\begin{cases}\min\{\Delta_{V_1}(\Delta_{V_2}-1), \Delta_{V_2}(\Delta_{{V_1}}-1)\} +1& \text{ if } \Delta(G)\geq 3\\
 3& \text{ if } \Delta(G)=2\end{cases}$$
 where $\Delta_{{V_1}}=\max\{d(u);~u\in V_1\}$ and   $\Delta_{{V_2}}=\max\{d(u);~u\in V_2\}$.
 \end{proposition}

The class of subcubic graphs is a class of graphs with many challenging conjectures. Proving results on subcubic graphs can give hints to prove results for  graphs  with any maximum degree.\\
The previous result implies that, for any subcubic bipartite graph $G$, $\chi_{\rm inj}'(G)\leq 7$.\\
We can decrease this bound by one.

 \begin{theorem}\label{cubHea}
 For any subcubic bipartite graph $G$, we have $\chi_{\rm inj}'(G)\leq 6$.
 \end{theorem}
We propose the following conjecture.
\begin{conjecture}
For every subcubic bipartite  graph, $\chi_{\rm inj}'(G)\leq 5$. 
 \end{conjecture}
 If the conjecture is true, then $5$ is tight, see for example the following subcubic bipartite graph (Figure \ref{BICUB}) which has an injective chromatic index equal to $5$.

\begin{figure}[ht]
\centering
\begin{pspicture}(0,-2.7971153)(7.5,2.7971153)
\psset{xunit=0.7,yunit=0.7}
\definecolor{colour1}{rgb}{0.6,0.0,0.6}
\definecolor{colour0}{rgb}{0.0,0.0,0.8}
\psline[linecolor=colour1, linewidth=0.04](6.9,-0.5999999)(6.9,1.4000001)
\psline[linecolor=colour0, linewidth=0.04](5.2,2.6000001)(0.5,1.4000001)
\psline[linecolor=green, linewidth=0.04](3.8,-0.5999999)(5.3,-0.5999999)
\psline[linecolor=yellow, linewidth=0.04](0.5,-0.5999999)(3.7,-2.5)(3.7,-2.5)
\psline[linecolor=red, linewidth=0.05](2.1,2.6000001)(2.1,-0.5999999)(2.1,-0.5999999)
\psdots[linecolor=black, dotsize=0.2](2.1,2.6000001)
\psdots[linecolor=black, dotsize=0.2](5.3,2.6000001)
\psdots[linecolor=black, dotsize=0.2](0.5,1.4000001)
\psdots[linecolor=black, dotsize=0.2](6.9,1.4000001)
\psdots[linecolor=black, dotsize=0.2](6.9,-0.5999999)
\psdots[linecolor=black, dotsize=0.2](5.3,-0.5999999)
\psdots[linecolor=black, dotsize=0.2](2.1,-0.5999999)
\psdots[linecolor=black, dotsize=0.2](0.5,-0.5999999)
\psdots[linecolor=black, dotsize=0.2](3.7,-0.5999999)
\psdots[linecolor=black, dotsize=0.2](3.7,-2.6)
\psline[linecolor=black, linewidth=0.04](3.7,-2.6)(3.7,-0.5999999)(3.7,-0.5999999)
\psline[linecolor=black, linewidth=0.04](5.3,2.6000001)(5.3,-0.5999999)(5.3,-0.5999999)
\rput[bl](2.9,1.8000001){$a$}
\rput[bl](2.2,1.0000001){$b$}
\rput[bl](4.1,-0.8999999){$c$}
\rput[bl](6.9,0.20000008){$d$}
\rput[bl](1.4,-1.6999999){$e$}
\rput[bl](1.0,2.1000001){$1$}
\rput[bl](4.1,2.4){$1$}
\rput[bl](0.0,0.50000006){$1$}
\rput[bl](6.4,0.6000001){$4$}
\rput[bl](6.0,2.0){$4$}
\rput[bl](5.7,1.3000001){$4$}
\rput[bl](2.1,0.20000008){$2$}
\rput[bl](2.5,-0.4999999){$2$}
\rput[bl](1.1,-0.4999999){$2$}
\rput[bl](5.3,0.100000076){$3$}
\rput[bl](5.6,-0.4999999){$3$}
\rput[bl](4.3,-0.4999999){$3$}
\rput[bl](2.2,-1.6999999){$5$}
\rput[bl](3.7,-1.8){$5$}
\rput[bl](5.2,-1.9){$5$}
\psline[linecolor=black, linewidth=0.04](0.5,1.3000001)(0.5,1.3000001)
\psline[linecolor=black, linewidth=0.04](0.5,1.4000001)(2.1,2.6000001)
\psline[linecolor=black, linewidth=0.04](2.2,2.6000001)(2.2,2.6000001)
\psline[linecolor=black, linewidth=0.04](2.1,2.6000001)(6.9,1.4000001)
\psline[linecolor=black, linewidth=0.04](5.3,2.6000001)(6.9,1.4000001)
\psline[linecolor=black, linewidth=0.04](0.5,1.4000001)(0.5,1.4000001)
\psline[linecolor=black, linewidth=0.04](0.5,1.4000001)(0.5,-0.5999999)(0.5,-0.5999999)
\psline[linecolor=black, linewidth=0.04](0.5,-0.5999999)(2.1,-0.5999999)(2.1,-0.5999999)
\psline[linecolor=black, linewidth=0.04](2.1,-0.5999999)(2.1,-0.5999999)
\psline[linecolor=black, linewidth=0.04](2.1,-0.5999999)(3.7,-0.5999999)(3.7,-0.5999999)
\psline[linecolor=black, linewidth=0.04](3.7,-0.5999999)(3.7,-0.5999999)
\psline[linecolor=black, linewidth=0.04](5.4,-0.5999999)(6.9,-0.5999999)(6.9,-0.4999999)
\psline[linecolor=black, linewidth=0.04](3.7,-2.5)(6.9,-0.5999999)(6.9,-0.5999999)
\end{pspicture}

\caption{A subcubic bipartite graph with an injective chromatic index equal to 5.}
\label{BICUB}

\end{figure}
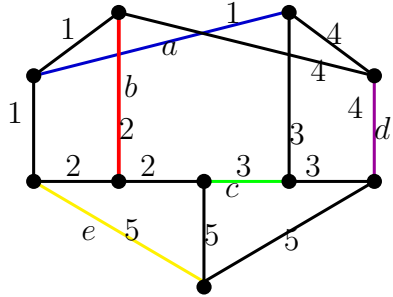
It is easy to see that we cannot do better than $5$ because in an injective edge-coloring the edges $a,b,c,d,e$ must have different colors.\\

\subsection*{Bounds for sparse subcubic graphs and subcubic outerplanar graphs}

  We give also  sufficient conditions for the injective chromatic index of a subcubic graph to be at most $4$ and $6$, and $7$ in terms of the maximum
average degree $\textrm{mad}(G)=\textrm{max}\left\{\frac{2|E(H)|}{|V(H)|},\;H \subseteq G\right\}$.

Note that the best possible sufficient condition for $3$ colors is $\textrm{mad}(G)<2$. If $\textrm{mad}(G)<2$ then $G$ is acyclic and by Proposition~\ref{CCCD}, $\chi_{\rm inj}'(G)\leq 3$. On the other hand, the graph $G$ in Figure \ref{FigSasha} has  $\textrm{mad}(G)=2$ and $\chi_{\rm inj}'(G)\geq 4$.

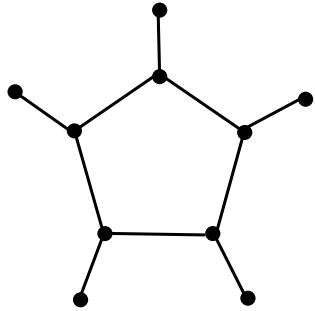
\begin{figure}[H]
\centering
{\begin{pspicture}(0,-2.06)(4.08,2.06)
\psline[linewidth=0.04cm](0.9,0.28)(0.14,0.82)
\psline[linewidth=0.04cm](1.28,-1.06)(0.96,-1.92)
\psline[linewidth=0.04cm](2.76,-1.1)(3.18,-1.94)
\psline[linewidth=0.04cm](3.9,0.76)(3.12,0.34)
\psline[linewidth=0.04cm](2.0,1.9)(2.02,1.04)
\psline[linewidth=0.04cm](1.94,1.04)(0.9,0.32)
\psline[linewidth=0.04cm](0.9,0.32)(1.26,-0.96)
\psline[linewidth=0.04cm](1.36,-1.04)(2.62,-1.06)
\psline[linewidth=0.04cm](2.78,-0.98)(3.14,0.34)
\psline[linewidth=0.04cm](3.08,0.34)(2.04,1.06)
\psdots[dotsize=0.2](2.02,1.04)
\psdots[dotsize=0.2](0.9,0.32)
\psdots[dotsize=0.2](1.3,-1.04)
\psdots[dotsize=0.2](2.72,-1.04)
\psdots[dotsize=0.2](3.14,0.3)
\psdots[dotsize=0.2](2.02,1.92)
\psdots[dotsize=0.2](3.94,0.74)
\psdots[dotsize=0.2](3.18,-1.9)
\psdots[dotsize=0.2](0.98,-1.92)
\psdots[dotsize=0.2](0.12,0.84)
\end{pspicture}}
\caption{A subcubic graph $G$ with ${\textrm mad}(G)=2$ and  $\chi_{\rm inj}'(G)=4$}
\label{FigSasha}
\end{figure}

In 2018, Bu and Qi \cite{BQ} proved:

\begin{theorem}\label{bu}\cite{BQ}
Let $G$ be a subcubic graph.
\begin{enumerate}
\item If ${\textrm mad}(G)<\frac{5}{2}$, then $\chi_{\rm inj}'(G)\leq 5$.
\item  If ${\textrm mad}(G)<\frac{18}{7}$, then $\chi_{\rm inj}'(G)\leq 6$.
\end{enumerate} 
\end{theorem}
 We improve the result of Theorem \ref{bu}.2 and we show that
\begin{theorem}\label{submad}
Let $G$ be a subcubic graph.
\begin{enumerate}
\item\label{mad1} If ${\textrm mad}(G)<\frac{7}{3}$, then $\chi_{\rm inj}'(G)\leq 4$.
\item\label{mad2}  If ${\textrm mad}(G)<\frac{8}{3}$, then $\chi_{\rm inj}'(G)\leq 6$.
\item\label{mad3}  If ${\textrm mad}(G)<3$, then $\chi_{\rm inj}'(G)\leq 7$.
\end{enumerate} 
\end{theorem}
We have to notice that the existing bounds for the strong edge-coloring of subcubic graphs proved in \cite{HOC13}, among other results, are the following:
\begin{theorem}
Let $G$ be a subcubic graph.
\begin{enumerate}
\item If ${\textrm mad}(G)<\frac{20}{7}$, then $\chi_{s}'(G)\leq 9$.
\item If ${\textrm mad}(G)<\frac{8}{3}$, then $\chi_{s}'(G)\leq 8$.

\item If ${\textrm mad}(G)<\frac{5}{2}$, then $\chi_{s}'(G)\leq 7$.

\item\ If ${\textrm mad}(G)<\frac{7}{3}$, then $\chi_{s}'(G)\leq 6$.
\end{enumerate} 
\end{theorem}
Since for any planar graph $G$ with girth at least $g$, we have $\text{mad}(G) <\frac{2g}{g-2}$, the following corollary can
be easily derived from Theorem \ref{submad} and Theorem \ref{bu}.2:
\begin{corollary}
Let $G$ be a subcubic planar graph with girth $g$.
\begin{enumerate}
\item If $g\geq16$, then $\chi_{\rm inj}'(G)\leq 4$.

\item If $g\geq10$, then $\chi_{\rm inj}'(G)\leq 5$ \cite{BQ}.

\item If $g\geq 8$, then $\chi_{\rm inj}'(G)\leq 6$.
\item If $g\geq 6$, then $\chi_{\rm inj}'(G)\leq 7$.
\end{enumerate} 
\end{corollary}

 We have also studied the injective chromatic index of subcubic outerplanar graphs. 
 An outerplanar graph is a graph with a planar drawing for which all vertices belong to the outer face of the drawing. 
 
 We prove the following result:

\begin{theorem}\label{AS4-1}
If $G$ is an outerplanar graph with maximum degree $\Delta(G)=3$, then $\chi_{\rm inj}'(G)\leq 5$. Moreover, the bound is tight.
\end{theorem}

\begin{figure}[!h]
\centering
\begin{pspicture}(0,-1.028125)(6.7628126,1.028125)
\psline[linewidth=0.04cm](2.7809374,0.0096875)(3.7009375,-0.0103125)
\psline[linewidth=0.04cm](3.7009375,-0.0103125)(4.3209376,0.5496875)
\psline[linewidth=0.04cm](3.7009375,0.0096875)(4.3209376,-0.6503125)
\psline[linewidth=0.04cm](4.3409376,0.5496875)(4.3209376,-0.6503125)
\psline[linewidth=0.04cm](4.3209376,0.5496875)(5.3009377,0.5496875)
\psline[linewidth=0.04cm](4.3409376,-0.6303125)(5.3009377,-0.6303125)
\psline[linewidth=0.04cm](5.3009377,0.5496875)(5.2809377,-0.6103125)
\psline[linewidth=0.04cm](5.3209376,0.5496875)(5.9409375,-0.0303125)
\psline[linewidth=0.04cm](5.9209375,-0.0303125)(5.3009377,-0.6303125)
\psline[linewidth=0.04cm](0.5609375,0.0096875)(1.1809375,0.5696875)
\psline[linewidth=0.04cm](0.5609375,0.0296875)(1.1809375,-0.6303125)
\psline[linewidth=0.04cm](1.2009375,0.5696875)(1.1809375,-0.6303125)
\psline[linewidth=0.04cm](1.1809375,0.5696875)(2.1609375,0.5696875)
\psline[linewidth=0.04cm](1.2009375,-0.6103125)(2.1609375,-0.6103125)
\psline[linewidth=0.04cm](2.1609375,0.5696875)(2.1409376,-0.5903125)
\psline[linewidth=0.04cm](2.1809375,0.5696875)(2.8009374,-0.0103125)
\psline[linewidth=0.04cm](2.7809374,-0.0103125)(2.1609375,-0.6103125)
\psdots[dotsize=0.18](2.7809374,0.0096875)
\psdots[dotsize=0.18](3.7809374,-0.0103125)
\psdots[dotsize=0.18](4.3409376,0.5296875)
\psdots[dotsize=0.18](5.3009377,0.5296875)
\psdots[dotsize=0.18](5.9009376,-0.0303125)
\psdots[dotsize=0.18](5.3009377,-0.6103125)
\psdots[dotsize=0.18](4.3409376,-0.6103125)
\psdots[dotsize=0.18](2.1809375,0.5496875)
\psdots[dotsize=0.18](2.1409376,-0.6103125)
\psdots[dotsize=0.18](1.1809375,-0.5903125)
\psdots[dotsize=0.18](1.2209375,0.5696875)
\psdots[dotsize=0.18](0.6409375,0.0096875)
\usefont{T1}{ptm}{m}{n}
\rput(2.8523438,0.2196875){$u$}
\usefont{T1}{ptm}{m}{n}
\rput(3.6923437,0.1996875){$v$}
\usefont{T1}{ptm}{m}{n}
\rput(2.3523438,0.7796875){$u_1$}
\usefont{T1}{ptm}{m}{n}
\rput(2.3323438,-0.7603125){$u_2$}
\usefont{T1}{ptm}{m}{n}
\rput(4.492344,0.7996875){$v_1$}
\usefont{T1}{ptm}{m}{n}
\rput(4.3923435,-0.7603125){$v_2$}
\usefont{T1}{ptm}{m}{n}
\rput(1.3323437,-0.8003125){$u_3$}
\usefont{T1}{ptm}{m}{n}
\rput(0.41234374,0.1796875){$u_4$}
\usefont{T1}{ptm}{m}{n}
\rput(1.2923437,0.8396875){$u_5$}
\usefont{T1}{ptm}{m}{n}
\rput(5.472344,-0.8003125){$v_3$}
\usefont{T1}{ptm}{m}{n}
\rput(6.2723436,-0.1203125){$v_4$}
\usefont{T1}{ptm}{m}{n}
\rput(5.6523438,0.7396875){$v_5$}
\end{pspicture} 
\caption{An outerplanar graph $G$ with $\Delta(G)=3$ and $\chi_{\rm inj}'(G)=5$}\label{chi-inj-5}
\end{figure}
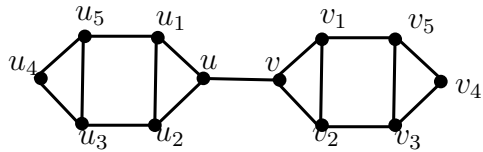
The graph $G$ depicted in Figure \ref{chi-inj-5} is an outerplanar graph with maximum degree 3 and $\chi'_{\rm {inj}}(G)=5$. It shows that the result of Theorem \ref{AS4-1} is tight.

 Indeed, since each pair of the set $\{uu_1,uu_2,vv_1,vv_2\}$ is either in a triangle or at distance two, in every injective edge-coloring $\phi$ of $G$, they must have four different colors. Let us then use colors 1,2,3,  and 4 to color respectively $uu_1$, $uu_2$, $vv_1$, and $vv_2$. However, four colors are not enough to color $G$, because if we have only  4 colors, then by symmetry, w.l.o.g., we may assume that $\phi(uv)=3$. Hence, $\phi(u_1u_2)=4$. In this case, we have to set, in the order, $\phi(u_3u_5)=3$, $\phi(u_4u_5)=2$, $\phi(u_3u_4)=1$, and $\phi(u_1u_5)=4$. Thus, no color can be given to $u_3u_2$. Consequently, $\chi_{\rm inj}'(G)\geq 5$. It is easy to find an injective edge-coloring of $G$ with 5 colors. 
 
\vspace{1cm}
The structure of the paper is as follows. In the next section we introduce some notation. In Section 3,  we prove Theorem \ref{AS2} and Theorem \ref{Reverse}. In Section 4, we prove Theorem 1.9. In Section 5, we prove Proposition 1.12 and Theorem \ref{cubHea}. In Section 6, we prove Theorem \ref{submad}. In Section 7, we prove Theorem \ref{AS4-1}. The last section is devoted to open problems. 
\section{Preliminaries} 

 Let $G$ be a graph, a vertex of degree $k$ is called a $k$-vertex, and a $k$-neighbor of a vertex $v$ is a $k$-vertex adjacent to $v$. We say a graph $G$ is subcubic if $\Delta(G)\leq 3$. The diameter of graph $G$ is the maximum distance between the pair of vertices.\\
Let $G$ be a graph, a $k$-injective edge-coloring of $G$ is an injective edge-coloring using the colors $\{1,\ldots,k\}$. \\

\section{Proof of Theorem \ref{AS2} and Theorem \ref{Reverse} }
In this section we prove:
\begin{TheoNonNum}
Let G be a graph with $\chi_{st}(G) = k$. Then $\chi_{\rm inj}'(G)\leq \frac{k(k-1)}{2}$.
\end{TheoNonNum} 
\begin{proof}
Let $G $ be a graph and $\phi$ a star coloring of $G$ with $k$ colors.\\
 Consider $H_{i,j}$ the graph induced by the union of two color classes $i$ and $j$ in $\phi$, for $i\neq j$ and
$i, j\in\{1, 2,\ldots, k\}$. Note that $H_{i,j}$ is a star forest. We have an injective edge-coloring of $H_{i,j}$ by coloring every edge of $H_{i,j}$ with the same color. After coloring each of the $\frac{k(k-1)}{2}$ graphs $H_{i,j}$, for $i\neq j$
and $i, j \in\{1, 2, \ldots , k\}$, we obtain an injective edge-coloring of G with $\frac{k(k-1)}{2}$ colors.

\end{proof}

We now prove  the reverse direction :

\begin{TheoNonNum}
Let $G$ be a graph with $\chi_{\rm inj}'(G)=k$. Then $\chi_{st}(G)\leq 2^k$.
\end{TheoNonNum}
\begin{proof}
Let $G$ be a graph and $C=\{1,2,\ldots,k\}$ be a set of $k$ colors. We may assume that $G$ is a connected graph and not an isolated edge. If $G$ is an isolated edge then is obviously true.
If $G$ is not a connected graph then we prove it for the connected components of $G$. \\

Assume that $G$ has an injective edge-coloring $\phi$ using $C$.
We define  a vertex coloring $\psi$ of $G$ as follows.
\begin{itemize}
\item For every $v\in V(G)$ and $d_G(v)\geq 2$, let $\psi(v)=(x_1,x_2,\ldots,x_k)$ such that, for every $i\in\{1,\ldots,k\}$,  
$$x_i=\begin{cases} 1 &\text{ if an edge of  color } i  \text{ is incident with the vertex } v;\\
0&\text{ otherwise}.
\end{cases}$$
\item For every $v\in V(G)$ and $d_G(v)=1$, let $\psi(v)=(0,0,\ldots,0)$.
\end{itemize}
 
The coloring $\psi$ is a proper vertex coloring. Indeed,  assume that  $\psi(v)=\psi(u)$, for some  $uv\in E(G)$. If $d_G(u)=1$ then $d_G(v)=1$ by the  definition of the coloring. It implies that $uv $ is an isolated edge, which contradicts the hypothesis. Now we assume that $d_G(u)\geq 2$ and $d_G(v)\geq 2$. Hence there exist two edges $uu_1$ and $vv_1$ in $G$ ($uu_1\neq vv_1$), such that $\phi(uu_1)=\phi(vv_1)$ which contradicts the coloring $\phi$. So it remains to show that there is no 2-colored path with four vertices in $G$.
 By contrary, assume that there exists a 2-colored path with four vertices $uvwx$ in $G$. Then $\psi(u)=\psi(w)$ and by definition of $\phi$, we can deduce that $\phi(uv)=\phi(vw)=\alpha$. In this case, it is easy to see that $\alpha\notin \phi(x)$, where $\phi(x)$ denote the set of colors used on the edges incident with $x$. Hence, $\psi(v)\neq\psi({x})$. We conclude that there is no 2-colored path with four vertices in $G$. Therefore, $\psi$ is a star coloring of $G$. The total number of used colors is at most $2^k$. 

\end{proof}
\begin{remark}
Let $G$ be a graph with minimum degree  at least 2. If $\chi_{\rm inj}'(G)=k$, then $\chi_{st}(G)\leq 2^k-1$.
\end{remark}
Since any vertex is incident with a least two colored edges, $\psi$ is not using the  color $(0,0,\ldots,0)$ ($G$ does not contain vertices of degree 1). The total number of used colors is at most $2^k-1$.
\section{Proof of Theorem 1.9}
An alternate way of looking at the injective chromatic index of a graph $G$ is to consider the graph $G^{(\ast)}$ obtained from $G$ as follows : $V(G^{(\ast)})=E(G)$ and two vertices of $G^{(\ast)}$ are adjacent if the edges of $G$ corresponding to these two vertices of $G^{(\ast)}$ are at distance 2 in $G$ or in a triangle ($G'$ is not exactly the $2$-distance graph of $G$). Then,  
\begin{equation}\label{inj-cubic}
\chi_{\rm inj}'(G)=\chi(G^{(\ast)})
\end{equation}
We recall that, by Proposition \ref{cyclechaine},
if $C_n$ is a cycle with $n$ vertices, then  
$\chi_{\rm inj}(C_n)= 2$ if $ n \equiv 0 \mod{4}$ and  $3$ otherwise.
We prove now the theorem:

\begin{TheoNonNum}\label{Delta}
Let $G$ be a graph with any maximum degree $\Delta(G)\geq 3$. Then $\chi_{\rm inj}'(G)\leq 2(\Delta(G)- 1)^2.$
\end{TheoNonNum}
\begin{proof}
Let $G=(V(G), E(G))$ be a graph with maximum degree $\Delta(G)$. Consider the graph $G^{(\ast)}$ defined above.

 Graph $G^{(\ast)}$ has a maximum degree at most $2(\Delta(G)- 1)^2$. We apply Brooks' theorem  to the components of $G^{(\ast)}$. $2(\Delta (G) - 1) ^ 2 $ colors are enough to properly color the vertices of the  components of $G^{(\ast)}$, except the case where  a component $ G^{(\ast)}$ is a complete graph $K_{2(\Delta (G) - 1) ^ 2+1}$ or an odd cycle. If no component of $G^{(\ast)}$ is  a complete graph, then the proof is done. 
If $G^{(\ast)}$ is an odd cycle (or an union of disjoint odd or even cycles), then $G$ is a cycle which is not possible by hypothesis. Now if a component of $G^{(\ast)}$ is a complete graph with degree  $ 2 (\Delta (G) - 1) ^ 2 $, then $G$ has $ 2 (\Delta (G) - 1) ^ 2 +1$ edges which are at distance $2$ from each other. Consider a vertex $s$ of  $K_{2(\Delta (G) - 1) ^ 2+1}$ and $e=uv$ the corresponding edge in $G$, then $e$ has exactly $2 (\Delta (G) - 1) ^ 2$ edges at distance $2$ in $G$ (see Figure \ref{BorneSup}). All these edges are represented by a vertex in  $K_{2(\Delta (G) - 1) ^ 2+1}$. Following the Figure \ref{BorneSup}, the edges $x_1z_i$, $i\in\{1,\ldots, \Delta(G)-1\}$, are represented  in  $K_{2(\Delta (G) - 1) ^ 2+1}$ and they are at distance $2$ from each other or in a triangle in $G$. Hence the set of vertices $z_i$, $i\in\{1,\ldots, \Delta(G)-1\}$ form a  complete graph $K_{\Delta (G)_1}$. In the same vein, $x_1z_1$ must be at distance two to $x_2t_1$ and $x_2t_2$. The only possibility is that in $G$ the edges $z_1t_1$ and $z_1t_2$ exist. It implies that $d(z_1)\geq \Delta(G) +1$, a contradiction.


 \begin{figure}[H]
  \begin{center}
\psfig{file=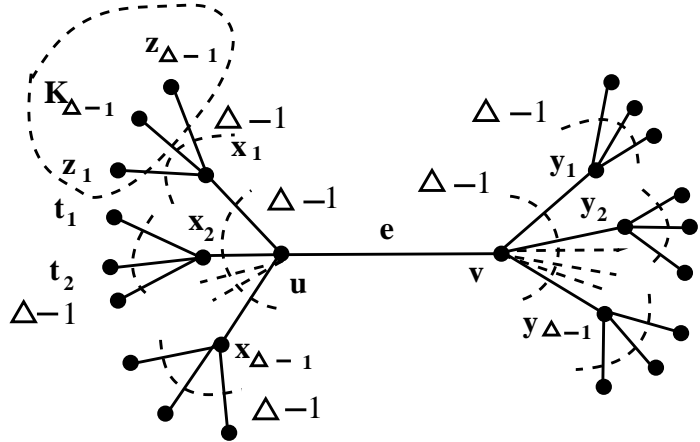,scale=.4}
\caption{Edge $e$ in $G$ corresponding to vertex $s$ in $G^{(\ast)}$}
\label{BorneSup}
\end{center}
\end{figure}
\end{proof}

\section{Proof of Proposition \ref{DeltaBiP} and Theorem \ref{cubHea} }

 \begin{propoNoNum}
 If $G=(V_1\cup V_2,E)$ is a bipartite graph with maximum degree $\Delta(G)$ and $G$ has no isolated vertices. Then $$\chi_{\rm inj}'(G)\leq\begin{cases}\min\{\Delta_{V_1}(\Delta_{V_2}-1), \Delta_{V_2}(\Delta_{{V_1}}-1)\} +1& \text{ if } \Delta(G)\geq 3\\
 3& \text{ if } \Delta(G)=2\end{cases}$$
 where $\Delta_{{V_1}}=\max\{d(u);~u\in V_1\}$ and   $\Delta_{{V_2}}=\max\{d(u);~u\in V_2\}$.
 \end{propoNoNum}
 \begin{proof}
 Let $G=(V_1\cup V_2,E)$ be a bipartite graph. We assume that $G$ is a connected graph, if not we consider each component.
   If $\Delta(G)=2$, then by Proposition \ref{cyclechaine} $\chi_{\rm inj}'(G)\leq 3$.
 
  Assume that $\Delta(G)\geq 3$. 
 Let  $G'=(V',E')$  be the graph obtained from $G$ as follows: $V'=V_1$ and two vertices of $G'$ are adjacent if they are at distance two. Hence, the maximum degree of $G'$ is at most $\Delta_{V_1}(\Delta_{V_2}-1)$. Since $G$ is connected graph, $G'$ is also a connected graph. We claim that $G'$ has a proper vertex coloring $\phi$ with  $\Delta_{V_1}(\Delta_{V_2}-1)$. Indeed, if $G'$ is not a  cycle or a complete graph, then by Brooks' theorem $\chi (G')\leq\Delta_{V_1}(\Delta_{V_2}-1)$; otherwise, if $G'$ is the complete graph $K_{\Delta_{V_1}(\Delta_{V_2}-1)+1}$, then $\chi_i(G')= \Delta_{V_1}(\Delta_{V_2}-1)+1$. Otherwise, if $G'$ is an odd cycle, then $\chi(G')= 3\leq \Delta_{V_1}(\Delta_{V_2}-1)+1$.

   We define an edge-coloring $\phi'$ of $G$  with $\Delta_{V_1}(\Delta_{V_2}-1)+1$ colors as follows: for any edge $uv$ of $G$ incident with a vertex $v\in V_1$, we set $\phi'(uv)=\phi(v)$. It is easy to see that for any three consecutive edges $e_1$, $e_2$, and $e_3$, we have $\phi'(e_1)\neq\phi'(e_3)$. Hence, $\phi'$ is an injective edge-coloring of $G$  and $\chi_{\rm inj}(G)\leq\Delta_{V_1}(\Delta_{V_2}-1)+1.$

Similarly, let  $G''=(V'',E'')$  be the graph obtained from $G$ as follows: $V''=V_2$ and two vertices of $G''$ are adjacent if they are at distance two. Hence, the maximum degree of $G''$ is at most $\Delta_{V_2}(\Delta_{V_1}-1)$. By a similar above argument, we can deduce that  $\chi_{\rm inj}(G)\leq\Delta_{V_2}(\Delta_{V_1}-1)+1.$ 

We conclude that  $\chi_{inj}'(G)\leq\min\{\Delta_{V_1}(\Delta_{V_2}-1), \Delta_{V_2}(\Delta_{{V_1}}-1)\}+ 1.$
\end{proof}
 We prove now the theorem for  subcubic bipartite graph :
 \begin{TheoNonNum}
 For any subcubic bipartite graph $G$, we have $\chi_{\rm inj}'(G)\leq 6$.
 \end{TheoNonNum}

\begin{proof}
  Let $G=(V_1\cup V_2,E)$ be a subcubic bipartite graph. If $\Delta(G)=2$ then by Proposition \ref{cyclechaine} $\chi_{\rm inj}'(G)\leq 3$.
 
  Assume that $\Delta(G)= 3$. 
 Let  $G'=(V',E')$  be the graph obtained from $G$ as follows: $V'=V_1$ and two vertices of $G'$ are adjacent if they are at distance two. Hence, the maximum degree of $G'$ is at most $6$.
 \begin{itemize}
 
  \item Since $G$ is not a cycle  $G'$ is not a cycle. If $G'$ is not a complete graph, then by Brooks' theorem  $\chi (G')\leq 6$.
  We define an edge coloring $\phi'$ of $G$  with $6$ colors as follows: for any edge $uv$ of $G$ incident with a vertex $v\in V_1$, we set $\phi'(uv)=\phi(v)$. It easy to see that for any three consecutive edges $e_1$, $e_2$ and $e_3$, we have $\phi'(e_1)\neq\phi'(e_3)$. Hence, $\phi'$ is an injective edge-coloring of $G$  and $\chi_{\rm inj}(G)\leq 6.$
 
   \item If $G'$ is the complete graph $K_7$, then we have: 
   \begin{enumerate}

   \item Graph $G$ is a cubic graph, it has 14 vertices and 21 edges.
   
   \item The diameter of $G$ is equal to $3$, 
  
   \item The girth of $G$ is equal to 6. 
    \end{enumerate}
 
 It implies that  G is a bipartite Moore graph  with $14$ vertices and $21$ edges. It is the unique (3,6)-cage: the Headwood graph (\cite{W82},\cite{BI73} Ch.23).
 We have an injective edge-coloring of the Heawood graph with $4$ colors (Figure \ref{HEAWOOD}), and it is easy to see that it is optimal.
 \begin{figure}[H]
  \begin{center}
\psfig{file=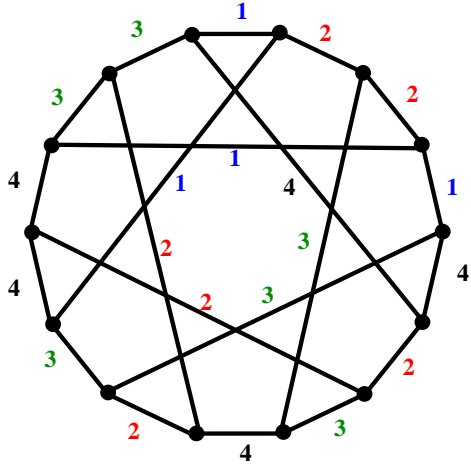,scale=.3}
\caption{An injective edge-coloring of the Heawood graph with $4$ colors.}
\label{HEAWOOD}
\end{center}
\end{figure}
 \end{itemize}
 This completes the proof.

 \end{proof}

\section{Proof of Theorem \ref{submad}}
For an edge-coloring $\phi$ of a graph $G$ and a vertex $v\in V(G)$, $\phi(v)$ denotes the set of colors used on the edges incident with $v$.

\subsection{Proof of Theorem \ref{submad}.\ref{mad1}}
We will say that an edge $uv$ is weak if at least one of $u$ and $v$ is a vertex of degree 1. A vertex $u$ is weak if at least one of the edges incident with $u$ is weak.\\

Let $C=\{1,2,3,4\}$ be a set of four colors. Suppose that Theorem \ref{submad}.\ref{mad1} is not true. Let $H$ be a counterexample minimizing $|E(H)| + |V (H)|$: H is
not injective edge-colorable with four colors, it satisfies $\textrm{mad}(H) <\frac{7}{3}$ and $\chi_{\rm inj}'(H\setminus e)\leq 4$ for any edge $e$.
\begin{lemma}\label{lem4cou}
The minimal counterexample $H$ satisfies the following properties:
\begin{enumerate}
\item $H$ does not contain a weak $2$-vertex.
\item $H$ does not contain a $3$-vertex adjacent to two $1$-vertices.
\end{enumerate}
\end{lemma}
\begin{proof}
\begin{enumerate}
\item Suppose that $H$ contains a 2-vertex $u$ adjacent to a 1-vertex $u_1$. Let $u_2$ be the second neighbor of $u
$. By  minimality of $H$, the graph $H' = H\setminus\{u_1u\}$ has an injective edge-coloring $\phi$ using $C$. We can view $\phi$ as a partial injective edge-coloring of $H$. We color $u_1u$ with a color $\alpha$ distinct from the colors of the edges adjacent to $uu_2$ (we have at most 2 forbidden colors), then we get an injective edge-coloring of $H$. Which is a contradiction.

\item Suppose that $H$ contains a 3-vertex $u$ with $N(u)= \{u_1, u_2, u_3\}$, where $d(u_1) = d(u_2) = 1$. By minimality of $H$, the graph $H' = H \setminus\{u_1u\}$ has an injective edge-coloring $\phi$ using $C$. We can view $\phi$ as a partial injective edge-coloring of $H$. We color $u_1u$ with a color $\alpha$ distinct from colors of edges adjacent to $uu_3$ (we have at most 2 forbidden colors). then we get an injective edge-coloring of $H$. We get a contradiction.
\end{enumerate}

\end{proof}

Let $H^\ast$ denote the graph obtained from $H$ by deleting all vertices of degree 1. If $H$ does not contain a 1-vertex, then $H^\ast = H$. Since $H^\ast\subseteq H$, $\textrm{mad}(H^\ast)<\frac{16}{7}$ and by the Lemma \ref{lem4cou}, $H^\ast$ does not contain a 1-vertex.
\begin{lemma}\label{C3C4P4}
$H^\ast$ satisfies the following properties:
\begin{enumerate}
\item \label{C3}$H^\ast$ has no 3-cycle $uvw$ such that $d_{H^\ast}(v) = d_{H^\ast} (w) = 2$ (Figure \ref{FigLem}.1).
\item \label{C4}$H^\ast$ has no cycle $xuvwx$ such that $d_{H^\ast}(u) =d_{H^\ast}(v) = d_{H^\ast} (w) = 2$ (Figure \ref{FigLem}.2).
\item \label{P4}$H^\ast$ has no path $xuvwy$ such that $d_{H^\ast}(u) =d_{H^\ast}(v) = d_{H^\ast} (w) = 2$ (Figure \ref{FigLem}.3).
\item\label{3V2V} $H^\ast$ does not contain a 3-vertex adjacent to three 2-vertices such that two of them has a 2-neighbor in $H^\ast$ (Figure \ref{FigLem}.4).
\end{enumerate}
\end{lemma}
\begin{proof}
\begin{enumerate}
\item Suppose that $H$ contains a 3-cycle $uvw$ such that  $d_{H^\ast}(v) = d_{H^\ast} (w) = 2$. If $x\in\{v,w\}$ has a $1-$neighbor in $H$, then we denote this neighbor by $x'$. Let $N_H(u)=\{w,v,y\}$.
Consider $H' = H \setminus\{v,w\}$. By minimality of $H$, $H'$ admits an injective edge-coloring $\phi$ using $C$. The coloring $\phi$ is a partial injective edge-coloring of $H$. We will extend $\phi$ to $H$. We color the edges $uv$, $uw$, $vw$, $vv'$, and $ww'$, in this order, as follows:
\begin{itemize}
\item $uv$ with any color $\alpha$ distinct from the colors of the edges adjacent to $uy$ (we have at most 2 forbidden colors).
\item $uw$ with any color $\beta$ distinct from $\alpha$ and distinct from the colors of the edges adjacent to $uy$ (we have at most 3 forbidden colors).
\item $vw$ with any color $\gamma$ distinct from $\alpha$, $\beta$, and $\phi(uy)$ (we have at most 3 forbidden colors).
\item $vv'$ with any color $\lambda$ distinct from $\beta$, and  $\phi(uy)$ (we have at most 2 forbidden colors).
\item $ww'$ with any color $\mu$ distinct from $\alpha$, $\lambda$, and  $\phi(uy)$ (we have at most 3 forbidden colors).
\end{itemize}
 Thus, we can extend $\phi$ to an injective edge-coloring of $H$, which is a contradiction.
\item  Suppose that $H$ contains a cycle $xuvwx$ such that $d_{H^\ast}(u) =d_{H^\ast}(v) = d_{H^\ast} (w) = 2$. If $s\in\{u,v,w\}$ has a 1-neighbor in $H$, then we denote this neighbor by $s'$, otherwise $s'$ does not exist. Let $y$ be the third neighbor of $x$ in $H$, if it exists. By the minimality of $H$, graph $H'=H\setminus\{v, u\}$ has an injective edge-coloring $\phi$ using $C$. We view $\phi$ as a partial injective edge-coloring of $H$. We will extend $\phi$ to $H$. We color the edges $xu$, $uv$, $vv'$, $uu'$, and $wv$, in this order, as follows:
\begin{itemize}
\item  $xu$ with any color $\alpha$ distinct from $\phi(ww')$ and distinct from the colors of the edges adjacent to $xy$ (we have at most 3 forbidden colors).
\item $uv$ with any color $\beta$ distinct from $\phi(xy)$, $\phi(xw)$, and $\phi(ww')$ (we have at most 3 forbidden colors).
\item  $vv'$ with any color $\gamma$ distinct from $\alpha$,  $\phi(xw)$, and $\phi(ww')$ (we have at most 3 forbidden colors).
\item  $uu'$ with any color $\lambda$ distinct from $\gamma$, $\phi(xy)$, and $\phi(xw)$ (we have at most 3 forbidden colors).
\item $wv$ with any color $\xi$ distinct from $\alpha$, $\lambda$, and $\phi(xy)$ (we have at most 3 forbidden colors).
\end{itemize}
So, we obtain an injective edge-coloring of $H$. A contradiction.

\item Assume now $H^\ast$ contains a path $xuvwy$ such that $d_{H^\ast}(u) =d_{H^\ast}(v) = d_{H^\ast} (w) = 2$. If $s\in\{u,v,w\}$ has a 1-neighbor in $H$, then we denote this neighbor by $s'$, otherwise $s'$ does not exist. Let $N_H(y)=\{w, y_1, y_2\}$ and $N_H(x)=\{u, x_1, x_2\}$. By the minimality of $H$, graph $H'= H\setminus\{uv,vw,uu',vv',ww'\}$ has an injective edge-coloring $\phi$ using $C$. We view $\phi$ as a partial injective edge-coloring of $H$. We will extend $\phi$ to $H$. First, we color the edges $uv$, $vw$ and $vv'$, in this order, as follows:
\begin{itemize}
\item $uv$ with any color $\alpha\notin\{\phi(xx_1),\phi(xx_2),\phi(wy)\}$.
\item  $vw$ with any color $\beta\notin\{\phi(yy_1),\phi(yy_2),\phi(xu)\}$.
\item $vv'$ with any color $\gamma\notin\{\phi(wy),\alpha,\phi(xu)\}$.
\end{itemize}
Hence, we get a partial injective edge-coloring $\phi_1$ of $H$. We will now color  $ww'$:
\begin{enumerate}
\item If there exists a color $\lambda$ distinct from $\phi_1(yy_1)$, $\phi_1(yy_2)$, $\alpha$, and $\gamma$ then, we set $\phi(ww')=\lambda$ and we color $uu'$ as follow:
\begin{enumerate}
\item If we can color $uu'$ with some color $\xi\notin\{\phi_1(xx_1),\phi_1(xx_2),\beta,\gamma\}$, then we get a contradiction.
\item Otherwise, w.l.o.g., we may assume that $\phi_11(vw)=2$, $\phi_1(vv')=3$, $\phi_1(xx_1)=1$, and $\phi_1(xx_2)=4$. In this case, by construction, $\alpha=2$ and $\alpha\neq \phi_1(xu)$. Hence, we permute the colors of  $uv$ and $ vv'$ and we assign color 3 to $uu'$. Thus, we get a contradiction.
\end{enumerate}
\item Otherwise, w.l.o.g., we may assume that $\alpha=1$, $\gamma=3$, $\phi(yy_1)=4$ and $\phi(yy_2)=2$.
\begin{enumerate}
\item If $\beta\neq\gamma$, then $\beta=1$  and $1\notin\{\phi'(wy),\phi'(xu)\}$, by construction. Hence,  we recolor $vv'$ with the color $1$ and assign color 3 to $ww'$. Next, we color $uu'$ with any color $\xi$ distinct from $1$, $\phi_1(xx_1)$, and $\phi_1(xx_2)$. Thus, we get a contradiction.

\item Otherwise, if $\beta=\gamma=3$, then $\phi_1(wy)\notin\{1,3\}$ and $1\notin\{\phi_1(xx_1),\phi_1(xx_2)\}$. 
\begin{itemize}
\item If $\phi_1(xu)=1$, then we uncolor $vv'$ and we set $\phi_1(uu')=1$, $\phi_1(ww')=3$ and we color $vv'$ with a color $\gamma'\notin\{\phi_1(wy),3,1\}$. Thus, we get a contradiction. 
\item If $\phi_1(xu)\neq 1$, then we uncolor the edges $vv'$ and $vw$. We set  $\phi_1(ww')=3$,  $\phi_1(vw)=\phi_1(vv')=1$, and $\phi_1(uu')=\xi'\notin\{\phi_1(xx_1),\phi_1(xx_2), 1\}$. Thus, we get a contradiction.
\end{itemize}

\end{enumerate} 
\end{enumerate}

\item Suppose that $H^\ast$ contains a $3$-vertex $u$ adjacent to three $2$-vertices $x$, $y$, and $z$ such that  $y$ (resp. $z$) has a $2$-neighbor $y_1$ (resp. $z_1$).
By Lemma \ref{C3C4P4}.\ref{C3} and Lemma \ref{C3C4P4}.\ref{C4}, the vertices  $x, y, z,y_1,z_1$ are paiwise distinct. Let $x_1$ (respectively, $t$ and $w$) denote the second neighbor in $H^\ast$ of $x$ (respectively, $y_1$ and $z_1$). For each $s\in\{x, y, z, x_1, y_1, z_1\}$, if $s$ has an  1-neighbor in $H$, then we denote this neighbor by $s'$.\\
Let $H' = H\setminus\{ux, uy, uz, xx_1, yy_1, zz_1, xx', x_1x_1', yy', y_1y_1', zz', z_1z_1'\}$. By the minimality of $H$, graph $H'$ has an injective edge-coloring $\phi$ using $C$. We view $\phi$ as a partial injective edge-coloring of $H$. We extend $\phi$ to $H$. We color in the order:
\begin{itemize}
\item  $ux$, $uy$, and $uz$ with the same color $\alpha$ distinct from $\phi(z_1t)$, $\phi(y_1w)$, and $\phi(x_1v)$.
\item $z_1z_1'$ and $zz_1$ with some color $\beta_1$ distinct from $\alpha$ and distinct from the colors of
the edges adjacent to $z_1t$ (we have at most 3 forbidden colors).
\item  $y_1y'_1$ and $yy_1$ with some color $\beta_2$ distinct from $\alpha$ and distinct from the colors of
the edges adjacent to $y_1w$ (we have at most 3 forbidden colors).
\item $x_1x_1'$ and $xx_1$ with some color $\beta_3$ distinct from $\alpha$ and distinct from the colors
of the edges adjacent to $x_1v$ (we have at most 3 forbidden colors).
\item   $zz'$ with some color $\gamma_1$ distinct from $\alpha$, $\beta_1$, and $\phi(z_1t)$.
\item $yy'$ with some color $\gamma_2$ distinct from $\alpha$, $\beta_2$, and $\phi(y_1w)$.
\item $xx'$ with some color $\gamma_3$ distinct from $\alpha$, $\beta_3$, and $\phi(x_1v)$.
\end{itemize}
We obtain an injective coloring of $H$.
Thus, we get a contradiction.
\end{enumerate}

\end{proof}
\begin{figure}[H]

\begin{center}
\psfig{file=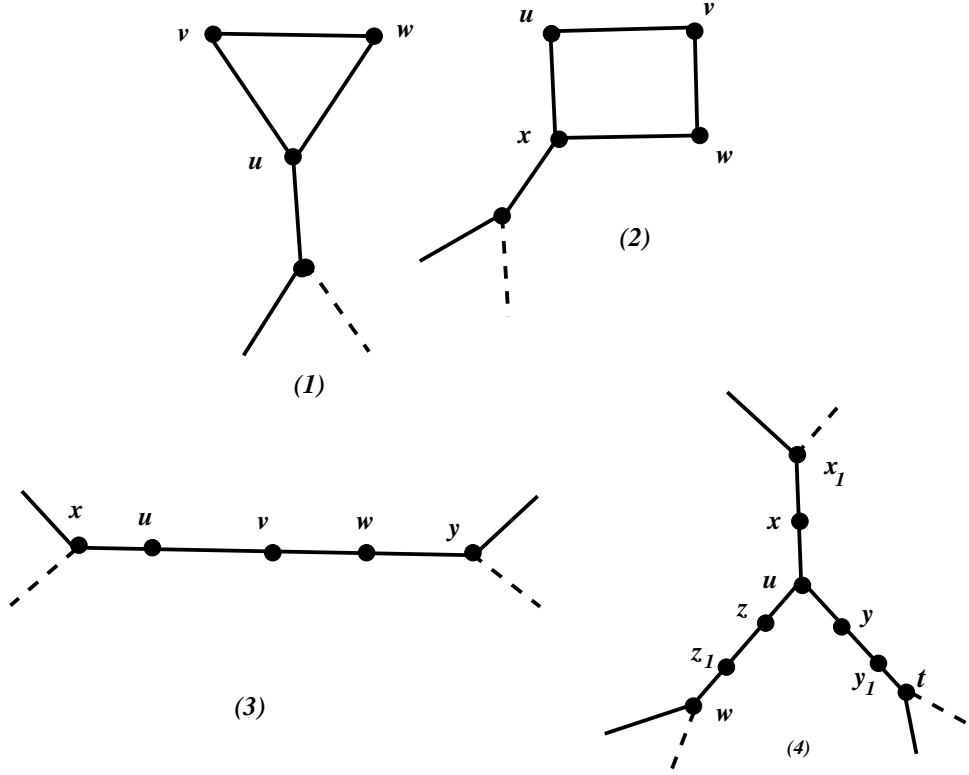,scale=.5}
\caption{Forbidden configurations.}
\label{FigLem}
\end{center}
\end{figure}
For $j\in\{1,2,3\}$, let $V_j$ be denote the set of vertices of degree $j$ in $H^\ast$ and $|V_j|=n_j$.\\ By Lemma \ref{lem4cou}, $n_1=0$. Let us call blue vertex a 2-vertex  adjacent to two 3-vertices and red vertex a 2-vertex adjacent to an other 2-vertex. Let $n_r$ the number of red vertices and $n_b$ the number of blue vertices. We have the maximum number of 2-vertices when each 3-vertex is adjacent to three 2-vertices.  By Lemma \ref{C3C4P4}.\ref{P4} and Lemma \ref{C3C4P4}.\ref{3V2V}, a 3-vertex is adjacent to at most two red vertices. So the maximum number of 2-vertices in the graph is reached when each 3-vertex is adjacent to two red vertices and one blue vertex. Hence in this case $n_b= n_3/2$ and $n_r= 2n/3$. It implies that
 $n_2\leq \frac{n_3}{2}+2n_3$. It follows that, 
\begin{equation}\label{equa11}
n_3\geq \frac{2n_2}{5}.
\end{equation}
Since $\textrm{mad}(H^\ast)\geq \frac{2|E(H^\ast)|}{|V(H^\ast)|}$, where  $2|E(H^\ast)|=3n_3+2n_2$, we get
 \begin{equation}\label{equa12}
\textrm{mad}(H^\ast)\geq 3- \frac{n_2}{n_3+n_2}.
\end{equation}
By Equation (\ref{equa11}), we may deduce that $\frac{n_2}{n_3+n_2}\leq\frac{5}{7}$, which yields $ \textrm{mad}(H^\ast)\geq \frac{7}{3}$. We get a contradiction. 

This completes the proof of Theorem \ref{submad}.\ref{mad1}.

\qed
\subsection{Proof of Theorem \ref{submad}.\ref{mad2}}
Let $C=\{1,2,3,4,5,6\}$ be a set of six colors. Suppose that Theorem \ref{submad}.\ref{mad2} is not true. Let $H$ be a counterexample minimizing $|E(H)| + |V (H)|$: H is
not injective edge-colorable with six colors, it satisfies $\textrm{mad}(H) <\frac{8}{3}$ and  $\chi_{\rm inj}'(H\setminus e)\leq 6$ for any edge $e$.
\begin{lemma}\label{lem6cou}
The minimal counterexample $H$ satisfies the following properties:
\begin{enumerate}
\item\label{lem6.1} $H$ does not contain a  $1$-vertex.
\item\label{lem6.2} $H$ does not contain two adjacent 2-vertices.
\item\label{lem6.3} $H$ does not contain a 3-vertex adjacent to two 2-vertices.
\end{enumerate}
\end{lemma}
\begin{proof}

\begin{enumerate}
\item Assume that $H$ has a 1-vertex $u$ adjacent to some $v$. Let $N_H(v)=\{u,u_1,u_2\}$. By the minimality of $H$, the graph $H\setminus\{u\}$ has an injective edge-coloring $\phi$ using $C$. The coloring $\phi$ is a partial injective edge-coloring of $H$. We extend $\phi$ to $H$ by coloring $uv$ with any color $\alpha$ distinct from the colors of the edges adjacent to $vu_1$ and $vu_2$ (we have at most 4 forbidden colors).

\item Suppose that $H$ contains a $2$-vertex $u$ adjacent to a $2$-vertex $v$. Let $t$ and $w$ be the other neighbors of $u$ and $v$ respectively. Let $N_H(w)=\{v,w_1,w_2\}$. By minimality of $H$, the graph $H' = H\setminus\{v\}$ has an injective edge-coloring $\phi$ using $C$. We extend $\phi$ to $H$. First, we color $uv$ with any color $\alpha$ distinct from the edges adjacent to $tu$ and $vw$ (we have at most 4 forbidden colors). Next, we color  $vw$ with any color $\beta$ distinct from $\phi(ut)$, and distinct from the edges adjacent to $ww_1$ and $ww_2$ (we have at most 5 forbidden colors).

\item Suppose $H$ contains a 3-vertex $u$ adjacent to two 2-vertices $v$ and $w$ whose second neighbors are $y$ and $z$, respectively. By minimality of $H$, there exists an injective edge-coloring $\phi$ of $H' = H \setminus\{v,w\}$ using $C$. Let $x$ be the third neighbor of $u$ and $N_H(y)=\{v,y_1,y_2\}$, $N_H(z)=\{w,z_1,z_2\}$. We will extend $\phi$ to $H$ by coloring, in the order, the following edges:
\begin{itemize}
\item  $vy$ with some color $\alpha$ different from $\phi(xu)$ and different from the colors of the edges adjacent to $yy_1$ and $yy_2$.
\item  $wz$ with some color $\beta$ different from $\phi(xu)$ and different from the colors of the edges adjacent to $zz_1$ and $zz_2$.
\item $uv$ with some color $\gamma$ different from $\beta$, $\phi(yy_1)$, $\phi(yy_2)$, and different from the colors of the edges adjacent to $ux$.
\item  $uw$ with some color $\lambda$ different from $\alpha$, $\phi(zz_1)$, $\phi(zz_2)$, and different from the colors of the edges adjacent to $ux$.
\end{itemize}
Thus, we can extend $\phi$ to $H$, which is a contradiction.

\end{enumerate}
\end{proof}

For $j\in\{1,2,3\}$, let $V_j$ be denote the set of vertices of degree $j$ in $H$ and $|V_j|=n_j$.\\ By Lemma \ref{lem6cou}.\ref{lem6.1}, $n_1=0$. By Lemma \ref{lem6cou}.\ref{lem6.2}, every $v\in V_2 $ has two neighbors in $V_3$. By Lemma \ref{lem6cou}.\ref{lem6.3}, every $v\in V_3$ has at most one neighbor in $V_2$. It follows that, 
\begin{equation}\label{equa21}
n_3\geq 2n_2.
\end{equation}
Since $\textrm{mad}(H)\geq \frac{2|E(H)|}{|V(H)|}$, where  $2|E(H)|=3n_3+2n_2$, we get
 \begin{equation}\label{equa22}
\textrm{mad}(H)\geq 3- \frac{n_2}{n_3+n_2}.
\end{equation}
By Equation (\ref{equa21}), we may deduce that $\frac{n_2}{n_3+n_2}\leq\frac{1}{3}$.
Thus, Equation (\ref{equa22}) yields $ \textrm{mad}(H)\geq \frac{8}{3}$. We get a contradiction.

This completes the proof of Theorem \ref{submad}.\ref{mad2}.

\qed

\subsection{Proof of Theorem \ref{submad}.\ref{mad3}}
Let $C =\{1,2,\ldots,7\}$ be a set of seven colors. Suppose to the contrary that Theorem \ref{submad}.\ref{mad3} is not true.  Let $H$ be a counterexample minimizing $|E(H)| + |V (H)|$: $H$ is not injective edge-colorable with seven colors,  it satisfies $\textrm{mad}(H) <3$ and $\chi_{\rm inj}'(H\setminus u)\leq 7$ for any vertex $u$.

Assume that $H$ has a 1-vertex $u$ adjacent to some $v$. Let $N_H(v)=\{u,u_1,u_2\}$. By the minimality of $H$, the graph $H\setminus\{u\}$ has an injective edge-coloring $\phi$ using $C$.  $\phi$ is a partial injective $7$-edge-coloring of $H$. We extend $\phi$ to $H$ by coloring $uv$ with any color $\alpha$ distinct from the edges adjacent to $vu_1$ and $vu_2$ (we have at most 4 forbidden colors). So, $\delta(H)\geq 2$.

Suppose now that $H$ has a $2$-vertex $u$ adjacent to $v$ and $w$. Let $N_H(v)=\{u, v_1, v_2\}$ and $N_H(w)=\{u,w_1,w_2\}$. By the minimality of $H$, graph $H\setminus\{u\}$ has an
injective edge-coloring $\phi$ using $C$. The coloring $\phi$ is a partial injective edge-coloring of $H$ using the colors of $C$. We will extend $\phi$ to $H$. First, we color the edge  $uv$ with any color $\alpha$ distinct from $\phi(ww_1)$, $\phi(ww_2)$, and distinct from the colors of the edges adjacent to $vv_1$ and $vv_2$ (we have at most 6 forbidden colors). Next, we color the edge $uw$ with any color $\beta$ distinct from $\phi(vv_1)$, $\phi(vv_2)$ and distinct from the colors of the edges adjacent to $ww_1$ and $ww_2$ (we have at most 6 forbidden colors). Thus, we can extend a coloring $\phi$ to $H$, which is a contradiction.

Therefore, $H$ is a $3$-regular graph which contradicts the hypothesis $\textrm {mad}(G)<3$.

\qed


\section{Proof of Theorem \ref{AS4-1}}

 A path $P=v_1v_2\ldots v_k$ is called a simple path in $G$ if $v_2,\ldots, v_{k-1}$ are all 2-vertices in $G$. The length of a path is the number of its edges.\\

 It is easy to see that the greedy coloring with $k + 1$ colors of a $k$-tree gives an acyclic coloring (see \cite{BLS}). Moreover, it is well-known that any outerplanar graph is a partial $2$-tree \cite{FRR}. Hence, if $G$ is an outerplanar graph, then $\chi_a(G)\leq 3$. By Theorem \ref{AS1} we have
\begin{corollary}
If $G$ is an outerplanar graph, then $\chi_{\rm inj}'(G)\leq 9$.
\end{corollary}
For the proof of Theorem \ref{AS4-1}, we will use the following folklore lemma (Proposition 20 in \cite{AMANDA}). 
\begin{lemma}\label{MR}
Every outerplanar graph $G$ with $\delta(G)\geq 2$ contains a cycle $xv_1\ldots v_kyx$, where $v_1,\ldots, v_k$ are
2-vertices, $k\geq1$, $d_G(x)\geq2$, and $d_G(y)\geq 2$.
\end{lemma}

We first prove a strong property for the easy case.
\begin{lemma}\label{Ci-Cj}
Let $G$ be the union of two cycles $C$ and $C'$ such that these two cycles have exactly one edge in common. Then $G$ has an injective 4-edge-coloring such that, in every simple path of length three, exactly two colors appear.
\end{lemma}

\begin{proof}
Let $G$ be the union of two cycles $C=xv_1\ldots v_iyx$ and $C'=xw_1\ldots w_jyx$ such that these two cycles have exactly one edge $xy$ in common. We define an injective 4-edge-coloring $\phi$ of $G$ with colors $\{\alpha,\beta,\gamma,\lambda\}$  such that in every simple path of length three has two colors as follows.
\begin{enumerate}
\item If $i\geq 2$ and $j\geq 2$, then we set first $\phi(xv_1)=\phi(xw_1)=\phi(xy)=\alpha$ and $\phi(yw_j)=\beta$. Next, we assign to the ordered edges $w_{j}w_{j-1},\ldots,w_2w_1$ the ordered colors  $\beta\gamma\gamma\beta\beta\gamma\gamma\beta\beta\ldots$ if $j$ is even  and $\gamma\gamma\beta\beta\gamma\gamma\beta\beta\ldots$ if $j$ is odd. Finally, we assign to the ordered edges $yv_i,v_iv_{i-1},\ldots,v_3v_2,v_2v_1$ the ordered colors $\lambda\gamma\gamma\lambda\lambda\gamma\gamma\ldots$ if $j$ is odd  and $\lambda\lambda\gamma\gamma\lambda\lambda\gamma\gamma\ldots$ if $j$ is even.

\item If $i=1$ and $j\geq 2$, then we set $\phi(xv_1)=\phi(xw_1)=\alpha$, $\phi(yv_1)=\beta$, and $\phi(xy)=\gamma$.
 Next, we assign to the ordered edges $yw_j,w_{j}w_{j-1},\ldots,w_2w_1$ the ordered colors  $\beta\lambda\lambda\beta\beta\lambda\lambda\beta\beta\ldots$ if $j$ is odd and  $\lambda\lambda\beta\beta\lambda\lambda\beta\beta\ldots$ if $j$ is even.
\item If $i=1$ and $j=1$, then we set $\phi(xv_1)=\phi(xw_1)=\alpha$, $\phi(yv_1)=\phi(yw_1)=\beta$, and $\phi(xy)=\gamma$.
\end{enumerate} 

\end{proof}

Now we prove the following stronger version of Theorem \ref{AS4-1}:
\begin{theorem}\label{AS4-2}
If $G$ is an outerplanar graph with maximum degree $\Delta(G)=3$, then $G$ has an injective 5-edge-coloring such that, in every simple path of length three, exactly two colors appear. 
\end{theorem}

\begin{proof}

Suppose to the contrary that Theorem \ref{AS4-2} is not true and let $G$ be a counterexample minimizing $|V (G)| + |E(G)|$. Graph $G$ is a connected graph.

\begin{claim}
$\delta(G)\geq 2$.
\end{claim}
\begin{proof}
Assume that $G$ has an edge, say $v_1v_2$, where $d_G(v_1)=1$. 
\begin{enumerate}
\item If $d_G(v_2)=3$, then consider the graph $H=G\setminus\{v_1v_2\}$. 
\begin{enumerate}
\item If $\Delta(H)=2$, then $H$ is a cycle $C=v_2v_3v_4\ldots v_iv_2$ or $H$ is a path. If $H$ is a path or triangle, then it is easy to see that $G$ has an injective edge-coloring using three colors, such that in every simple path of length three  exactly two colors appear, which is a contradiction. So, we may assume that $H$ is a cycle of order at least 4.   In this case, we define an injective 5-edge-coloring $\phi$ of $G$ as follows. First, we set $\phi(v_2v_3)=\phi(v_3v_4)=\alpha$, $\phi(v_1v_2)=\beta$, where $\beta\neq \alpha$. Next, we assign to the ordered edges $v_4v_5,v_5v_6,\ldots,v_{i-1}v_i,v_iv_2$ the ordered colors $\gamma\gamma\lambda\lambda\gamma\gamma\lambda\lambda\ldots$ and we get a contradiction. 
\item If $\Delta(H)=3$, then by minimality of $G$, $H$ has an injective 5-edge-coloring $\phi$ such that in every simple path of length three, exactly two colors appear. Let $N_G(v_2)=\{v_1,w_1,w_2\}$. In this case, we extend $\phi$ to $G$ by setting $\phi(v_1v_2)=\alpha$, such that $\alpha$ is distinct from the colors of the edges adjacent to the edges $v_2w_1$ and $v_2w_2$, we get a contradiction.    
\end{enumerate}
\item If $d_G(v_2)\leq 2$, then consider the maximal simple path $P=v_1v_2\ldots v_k$ in $G$, $k\geq 3$. Since $P$ is a maximal simple path and $\Delta(G)=3$, we have $d_G(v_k)=3$. Let $N(v_k)=\{v_{k-1},w_1,w_2\}$. Consider the graph $H=G\setminus\{v_1,\ldots,v_{k-2}\}$. By minimality of $G$, $H$ has an injective 5-edge-coloring $\phi$ such that, in every simple path of length three, exactly two colors appear. In this case, we assign to the ordered edges $v_{k-1}v_{k-2}, v_{k-2}v_{k-3},\ldots,v_2v_1$ the ordered colors $\alpha\alpha\beta\beta\alpha\alpha\beta\beta\alpha\alpha \ldots$, where $\alpha \neq \beta$  and  $\alpha\notin\{\phi(v_kw_1),\phi(v_kw_2),\phi(v_kv_{k-1})\}$, we get a contradiction.  
\end{enumerate}
\end{proof}

In the rest of the proof, we prove that such a configuration given by Lemma \ref{MR} can be reduced.
\begin{claim}
Graph $G$ does not contain a cycle $xv_1\ldots v_kyx$, where $v_1,\ldots, v_k,y$ are 2-vertices, $k\geq1$, $d_G(x)=3$.
\end{claim}
\begin{proof}
Assume that $G$ contains a cycle $xv_1\ldots v_kyx$, where $v_1,\ldots, v_k,y$ are 2-vertices, $k\geq1$, $d_G(x)=3$.
 Let $N_G(x)=\{y,v_1,x_1\}$.
\begin{enumerate}
\item If $k=1$, then by minimality of $G$, the graph $H=G\setminus\{v_1y\}$  has an injective 5-edge-coloring $\phi$ such that, in every simple path of length three, exactly two colors appear.  We can extend the coloring $\phi$ as follows. First, we recolor the edges $xy$ with a color $\alpha$, such that  $\alpha$ is distinct from the colors of the edges adjacent to $xx_1$ and distinct from $\phi(xv_1)$. Next, we set $\phi(yv_1)=\beta$, where $\beta\notin\{\alpha,\phi(xv_1),\phi(xx_1)\}$. Hence, we can extend $\phi$ to $G$ and we get a contradiction.
\item If $k\geq2$, then by minimality of $G$, the graph $H=G\setminus\{v_1v_2,\ldots,v_{k-1}v_k,v_ky\}$  has an injective 5-edge-coloring $\phi$ such that, in every simple path of length three, exactly two colors appear. We can extend the coloring $\phi$ to $G$ as follows. First, we recolor the edges $xv_1$ and $xy$ with the colors $\alpha$ and $\beta$ respectively such that  $\alpha\notin \phi(x_1)$ and $\beta\notin\phi(x_1)\cup\{\alpha\}$. Next, we set $\phi(yv_k)=\beta$ and we assign to the ordered edges $v_1v_2,v_2v_3,v_3v_4,\ldots,v_{k-2}v_{k-1}, v_{k-1}v_{k}$ the ordered colors $\alpha\gamma\gamma\alpha\alpha\gamma\gamma\alpha\alpha\gamma\gamma\ldots$ if $k$ is even and $\gamma\gamma\alpha\alpha\gamma\gamma\alpha\alpha\gamma\gamma\ldots$ if $k$ is odd, such that $\gamma\notin\phi(x)$. Hence, we can extend $\phi$ to $G$ and  we get a contradiction.
\end{enumerate}
Hence, in each case, we can extend the coloring to $G$, we get a contradiction.

\end{proof}

\begin{claim}
Graph $G$ does not contain a cycle $xv_1\ldots v_kyx$, where $v_1,\ldots, v_k$ are 2-vertices, $k\geq1$, $d_G(x)=d_G(y)=3$.
\end{claim}
\begin{proof}
Assume that $G$ contains a cycle $C=xv_1\ldots v_kyx$, where $v_1,\ldots, v_k$ are 2-vertices, $k\geq1$, $d_G(x)=d_G(y)=3$. Let $N_G(x)=\{y,v_1,x_1\}$ and $N_G(y)=\{x,v_k,y_1\}$. Consider the graph $H=G\setminus\{v_1,\ldots,v_k\}$, $k\geq1$. We claim that $\Delta(H)=3$. Indeed, if $\Delta(H)=2$, then $G$ is a union of two cycles $C_i$ and $C_j$ such that they have exactly one edge in common. Therefore, by  Lemma \ref{Ci-Cj}, $G$ has an injective 4-edge-coloring such that, in every simple path of length three, exactly two colors appear, which contradicts the choice of $G$. Hence, $\Delta(H)=3$ and by minimality of $G$, $H$ has an injective edge-coloring $\phi$ satisfying the hypothesis. We will extend $\phi$ to an injective edge-coloring of $G$ with the desired property as follows.
\begin{enumerate}
\item[\textbf{Case 1:}] Assume that $x_1= y_1$ and  denote $x_2$ the third neighbor of $x_1$. Let  $\alpha$, $\beta$, $\gamma$ be the colors of the edges $x_1x$, $xy$,  and $yx_1$ respectively.
\begin{enumerate}
\item If $k=1$, then we color $xv_1$ with a color $\ a\notin\{\beta,\gamma,\phi (x_1x_2)\}$ and we color $yv_1$ with a color $b\notin\{\alpha,\beta,a,\phi(x_1x_2)\}$. Hence, we get a contradiction.

\item If $k\geq2$, then consider the following two cases:
\begin{itemize}
\item If $k$ is even, then we assign to the ordered edges $xv_1,v_1v_2,v_2v_3,\ldots,v_{k-1}v_k$ the ordered colors $\lambda\lambda\alpha\alpha\lambda\lambda\alpha\alpha\lambda\lambda\ldots$, where $\lambda\notin\phi(x_1)\cup\{\beta\}$. The new partial edge-coloring $\phi'$ is an injective 5-edge-coloring such that, in every simple path of length three, exactly two colors appear. Now, we set $\phi(v_ky)=\xi$ ($\xi\notin\{\phi(x_1x_2),\alpha,\lambda\}$). Hence we can extend the coloring $\phi$ to $G$ and we get a contradiction.
\item
\begin{sloppypar}
 If $k$ is odd then, we set first $\phi(xv_1)=\beta$ and $\phi(v_{k}y)=\xi$  ($\xi\notin\{\alpha,\beta,\gamma,\phi(x_1x_2)\}$). Next we assign to the ordered edges $v_1v_2,v_2v_3,\ldots,v_{k-1}v_{k}$ the ordered colors $\lambda\lambda\alpha\alpha\lambda\lambda\alpha\alpha\lambda\lambda\ldots$, where $\lambda\notin\{\alpha,\beta,\gamma,\xi\}$. 
\end{sloppypar}

 Hence we can extend the coloring $\phi$ to $G$ and we get a contradiction.
\end{itemize}

\end{enumerate}

\item[\textbf{Case 2:}] Assume that $x_1\neq y_1$. 
W.l.o.g suppose that $\phi(xx_1)=\phi(xy)=\alpha$ and $\phi(yy_1)=\beta$, $\alpha\neq\beta$.
\begin{enumerate}
\item If $k=1$, then we can extend the coloring by setting $\phi(xv_1)=\gamma$ and $\phi(yv_1)=\lambda$, such that $\gamma$ is distinct from the colors colors of the edges ajacent with $xx_1$ and distinct from $\alpha$ and $\beta$, and $\lambda$ distinct from colors of edges adjavent with $yy_1$ and distinct from $\alpha$ and $\gamma$. Hence, we get a contradiction.

\item If $k\geq2$, then we color $xv_1$ with the color $\alpha$ and we consider the following cases:
\begin{itemize}
\item If $k$ is even, then we assign to the ordered edges $yv_k,v_kv_{k-1},\ldots,v_3v_2,v_2v_1$ the ordered colors $\gamma\gamma\beta\beta\gamma\gamma\beta\beta\gamma\gamma\ldots$, where $\gamma\notin\phi(y_1)\cup\{\alpha\}$. Hence we can extend the coloring $\phi$ to $G$ and we get a contradiction.
\item If $k$ is odd, then we assign to the ordered edges $v_kv_{k-1},v_{k-1}v_{k-2}\ldots,v_3v_2,v_2v_1$ the ordered colors $\gamma\gamma\lambda\lambda\gamma\gamma\lambda\lambda\gamma\gamma\ldots$, where $\gamma\notin\{\alpha,\beta\}$ and $\lambda\notin\{\alpha,\gamma\}$. The new partial edge-coloring $\phi'$ is an injective 5-edge-coloring such that, in every simple path of length three, exactly two colors appear. Now, we color $v_ky$ with some color $\xi$ such that $\xi$ is distinct from the color of the edges at distance exactly two (we have 4 forbidden colors). Hence we can extend the coloring $\phi$ to $G$ and we get a contradiction.
\end{itemize} 

\end{enumerate} 

\end{enumerate}


This completes the proof of Theorem \ref{AS4-2}.

\end{proof}

Then Theorem \ref{AS4-1} is proved.

\end{proof}

\section{Open problem}
 We summarize the conjectures we mentioned in the introduction.\\
 
  While there exist cubic graphs with the injective chromatic index equal to 6 (Figure \ref{fig6}), no example of a subcubic graph that would require 7 colors is known, we proposed the following conjecture.
\begin{conjecture}
For every subcubic graph, $\chi_{\rm inj}'(G)\leq 6$. 
\end{conjecture}
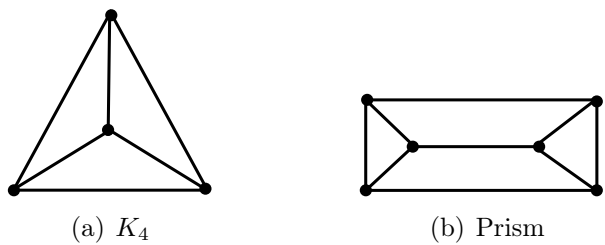
\begin{figure}[ht]
\centering
\subfigure[$K_4$]{
\begin{pspicture}(0,-1.26)(2.7,1.26)
\psdots[dotsize=0.12](1.32,-0.38)
\pstriangle[linewidth=0.04,dimen=outer](1.35,-1.18)(2.62,2.42)
\psline[linewidth=0.04cm](1.34,1.2)(1.32,-0.38)
\psline[linewidth=0.04cm](1.32,-0.38)(0.06,-1.16)
\psline[linewidth=0.04cm](1.3,-0.34)(2.62,-1.16)
\psdots[dotsize=0.16](1.36,1.16)
\psdots[dotsize=0.16](1.32,-0.36)
\psdots[dotsize=0.16](2.6,-1.14)
\psdots[dotsize=0.16](0.08,-1.16)
\end{pspicture} }
\quad\quad\quad\quad
\subfigure[Prism]{\begin{pspicture}(0,-0.7)(3.22,0.7)
\psframe[linewidth=0.04,dimen=outer](3.14,0.62)(0.06,-0.62)
\psline[linewidth=0.04cm](0.7,-0.02)(2.3,-0.02)
\psline[linewidth=0.04cm](3.1,0.6)(2.3,-0.02)
\psline[linewidth=0.04cm](2.3,-0.02)(3.08,-0.58)
\psline[linewidth=0.04cm](0.06,0.62)(0.7,0.0)
\psline[linewidth=0.04cm](0.7,-0.02)(0.08,-0.6)
\psdots[dotsize=0.16](3.12,0.58)
\psdots[dotsize=0.16](3.12,-0.6)
\psdots[dotsize=0.16](2.36,-0.02)
\psdots[dotsize=0.16](0.7,-0.02)
\psdots[dotsize=0.16](0.1,0.6)
\psdots[dotsize=0.16](0.08,-0.6)
\end{pspicture} 
}
\caption{Two graphs having  an injective chromatic index equal to 6}
\label{fig6}
\end{figure}
For the case of subcubic bipartite graph we made the following conjecture.
\begin{conjecture}
For every subcubic bipartite  graph, $\chi_{\rm inj}'(G)\leq 5$. 
 \end{conjecture}
 It is mentioned that there exist a cubic bipartite graph with injective chromatic index equal to $5$. \\
 
It would be interesting to understand the list version of injective edge-coloring. An edge list $L$ for a graph $G$ is a mapping that assigns a finite set of colors to each edge of $G$. Given an edge list $L$ for a graph $G$, we shall say that $G$ is injective edge-choosable if it has an injective edge-coloring $c$ such that $c(e)\in L(e)$ for every edge of $G$.
The list-injective chromatic index, $ch'_{inj}(G)$, of a graph $G$ is the minimum $k$ such that for every edge list $L$ for $G$ with $|L(e)| = k$ for every $e \in E(G)$, $G$ is L-injective-edge colorable. Let us ask the following question.
\begin{question}
Is it true that $ch'_{inj}(G)=\chi_{\rm inj}'(G)$ for every graph $G$?
\end{question}


\end{document}